\documentclass[a4paper,11pt]{article}
\usepackage[T1]{fontenc}
\usepackage{lmodern} 
\usepackage[utf8]{inputenc}
\usepackage[english]{babel}
\usepackage[margin=0.8in]{geometry}
\usepackage{microtype}
\usepackage{braket}
\usepackage{amsmath}
\usepackage{amssymb}
\usepackage{amsthm, aliascnt}
\usepackage{mathtools}
\usepackage{mathrsfs}
\usepackage{xcolor}
\usepackage[hyphens]{url}
\usepackage{hyperref}
\usepackage{enumitem}
\usepackage{tikz}
\usepackage[maxbibnames=99]{biblatex}
\usepackage{csquotes}
\usetikzlibrary{patterns}
\usepackage{tikz}
\usepackage{esint}
\usepackage{mleftright}

\newcommand{\al}{\alpha}
\newcommand{\la}{\lambda}
\newcommand{\si}{\sigma}
\newcommand{\La}{\Lambda}

\newcommand{\De}{\Delta}
\newcommand{\na}{\nabla}
\newcommand{\be}{\beta}
\newcommand{\ga}{\gamma}
\newcommand{\Ga}{\Ga}
\newcommand{\vp}{\varphi}

\newcommand{\pa}{\partial}

\newcommand{\ml}{\mleft}
\newcommand{\mr}{\mright}

\usepackage{xpatch}
\makeatletter
\newcounter{proofpart}
\xpretocmd{\proof}{\setcounter{proofpart}{0}}{}{}
\newcommand{\proofpart}[1]{%
  \par
  \addvspace{\medskipamount}%
  \stepcounter{proofpart}%
  \noindent\emph{Step \theproofpart: #1}\par\nobreak\smallskip
  \@afterheading
}
\makeatother
\hypersetup{
                colorlinks=true,
                linkcolor={magenta},
                citecolor={cyan},
                breaklinks=true}

\theoremstyle{plain}
\newtheorem{teor}{Theorem}
\numberwithin{teor}{section}
\numberwithin{equation}{section}

\theoremstyle{definition}
\newaliascnt{defi}{teor}

\aliascntresetthe{defi}

\theoremstyle{plain}
\newaliascnt{lemma}{teor}
\newtheorem{lemma}[lemma]{Lemma}
\aliascntresetthe{lemma}
\theoremstyle{plain}
\newaliascnt{prop}{teor}
\newtheorem{prop}[prop]{Proposition}
\aliascntresetthe{prop}
\theoremstyle{plain}
\newaliascnt{conjecture}{teor}

\aliascntresetthe{conjecture}
\theoremstyle{plain}
\newaliascnt{cor}{teor}

\aliascntresetthe{cor}
\theoremstyle{definition}
\newaliascnt{ex}{teor}

\aliascntresetthe{ex}
\theoremstyle{definition}
\newaliascnt{oss}{teor}
\newtheorem{oss}[oss]{Remark}
\aliascntresetthe{oss}
\theoremstyle{plain}

\DeclarePairedDelimiter{\abs}{\lvert}{\rvert}
\DeclarePairedDelimiter{\norma}{\lVert}{\rVert}

\newcommand{\R}{\mathbb{R}}

\newcommand{\N}{\mathbb{N}}

\DeclareMathOperator{\supp}{supp}

\title{On the buckling eigenvalue of unbounded cylinders}
\author{Paolo Acampora, Emanuele Cristoforoni, Carlo Nitsch, Cristina Trombetti}
\date{}

\newcommand{\Addresses}{{
 \bigskip 
 \footnotesize 

  \textsc{Dipartimento di Bioscienze e Territorio,
  Università degli Studi del Molise, Viale dell'universit\`a 3, 86090, Pesche (IS)
  .}\par\nopagebreak 

  \medskip 
 
 \textit{E-mail address}, P.~Acampora: \texttt{paolo.acampora@unimol.it} 

    \bigskip

        \textsc{Dipartimento di Matematica ``Federigo Enriques'', Universit\'a degli Studi di Milano La Statale, Via Saldini 50 20123 Milano Italia.}\par\nopagebreak 
 
 \medskip 
 
 \textit{E-mail address}, E.~Cristoforoni: \texttt{emanuele.cristoforoni@unimi.it} 

 \bigskip
 
 \textsc{Dipartimento di Matematica e Applicazioni ``R. Caccioppoli'', Universit\`a degli studi di Napoli Federico II, Via Cintia, Complesso Universitario Monte S. Angelo, 80126 Napoli, Italy.}\par\nopagebreak

 \medskip 
 
 \textit{E-mail address}, C.~Nitsch: \texttt{c.nitsch@unina.it}

  \medskip 
 
 \textit{E-mail address}, C.~Trombetti: \texttt{cristina@unina.it} 
 
}}

\begin{document}

\maketitle

\begin{abstract}
       We variationally characterize the bottom of the spectrum of the buckling problem in an infinite cylinder $A\times\R^{n-d}$, where $A$ is an open bounded subset of $\R^d$, and compute it explicitly when $A$ is a ball.    
       

\end{abstract}
\section{Introduction}

For every open set $\Omega\subset\mathbb{R}^n$, the bottom of the spectrum for the buckling eigenvalue problem is defined variationally as
$$ \displaystyle \Lambda(\Omega)=\inf_{v\in C^\infty_c(\Omega)\setminus\{0\}}  \frac{\displaystyle\int_\Omega (\Delta v)^2\,dx}{\displaystyle\int_\Omega |\nabla v|^2\,dx}. $$
When $\Omega$ is bounded, $\Lambda(\Omega)$ corresponds exactly to the principal eigenvalue of the fourth-order boundary value problem
\[ \begin{cases}
    \Delta^2 v = -\Lambda \Delta v &\text{in }\Omega,\\[5pt]
    v=\partial_\nu v = 0 &\text{on }\partial \Omega,
\end{cases} \]
where $\partial_\nu v$ denotes the outer normal derivative on the boundary $\partial \Omega$. Similarly, we recall that the bottom of the spectrum for the Dirichlet Laplacian is given by
\[ \lambda(\Omega)=\inf_{u\in C^\infty_c(\Omega)\setminus\{0\}} \frac{\displaystyle\int_\Omega |\nabla u|^2\,dx}{\displaystyle\int_\Omega u^2\,dx }. \]

In this paper, we completely characterize the bottom of the spectrum of the buckling problem for unbounded cylinders of the form $\Omega = A\times\mathbb{R}^{n-d}$, where $A \subset \mathbb{R}^d$ is a bounded open set. It is a well-known fact that for the Dirichlet problem, the bottom of the spectrum in a cylinder reduces precisely to the principal eigenvalue of its bounded section; that is, $\lambda(A\times\mathbb{R}^{n-d})=\lambda(A)$. Strikingly, the analogous property does not hold for the buckling problem.

To formulate our main result, for every $\mu\ge 0$ let $\Lambda_\mu(A)$ be the first eigenvalue of the modified boundary value problem
\[ \begin{cases}
     (\Delta-\mu)^2 u = - \Lambda (\Delta-\mu)u &\text{in }A,\\[5pt]
    u=0 &\text{on }\partial A,\\[5pt]
    \partial_\nu u=0 &\text{on }\partial A,
\end{cases} \]
which is variationally characterized by
\[ \Lambda_\mu(A)=\min_{h\in H^2_0(A)\setminus\{0\}} \frac{\displaystyle\int_A(\Delta h -\mu h)^2\,dx}{\displaystyle\int_A|\nabla h|^2+\mu h^2\,dx}. \]
Our first theorem establishes that the buckling eigenvalue of a cylinder is obtained by minimizing $\Lambda_\mu(A)$ over all possible frequency shifts $\mu \ge 0$.

\begin{teor}\label{main}
    Let $0<d<n$ and let $A\subset\mathbb{R}^d$ be a bounded open set. Then
    \[ \Lambda(A\times\mathbb{R}^{n-d})=\min_{\mu\ge0}\Lambda_\mu(A). \]
\end{teor}

Furthermore, when the section $A$ is a ball, we can compute this minimum explicitly. We show that the buckling eigenvalue of the cylinder reduces to the eigenvalue of its section if and only if the dimension of the section is strictly greater than one.

\begin{teor}\label{main2}
  Let $0<d<n$ and let $B_1^d\subset\mathbb{R}^d$ be the unit ball in $\mathbb{R}^d$. Then
  \[ \Lambda(B_1^d \times \mathbb{R}^{n-d}) = \Lambda (B_1^d) \]
  if and only if $d>1$. Moreover, if $d=1$, then
  \[ \Lambda((-1,1) \times \mathbb{R}^{n-1}) = \Lambda_\mu(-1,1) \]
  where $\mu$ is the unique positive solution to
    \begin{equation}
   \label{eq:optimalMu}
  \sqrt\mu\,\tanh(\sqrt\mu)=1.
    \end{equation}
\end{teor}

The investigation of the buckling eigenvalue on unbounded cylinders is fundamental in the study of the classical Dirichlet-buckling inequality. In \cite{Payne1955} (see also \cite{Payne}), Payne proved that for every bounded convex set $\Omega\subset\mathbb{R}^n$, the Dirichlet and buckling eigenvalues are related by the inequality
\begin{equation}\label{Payneineq}
    \Lambda(\Omega)\le 4\lambda(\Omega).
\end{equation}
Looking for optimal sequences achieving equality in~\eqref{Payneineq} it is natural to compute the eigenvalues on a sequence of boxes $\Omega_k=(0,1)\times(-k,k)^{n-1}$ converging to the infinite strip $(0,1)\times \mathbb{R}^{n-1}$. It was thought that the infinite strip would have asymptotically saturated the inequality, implying that the universal constant $4$ could not be improved (see, for instance, \cite{Payne1955} or \cite[Remark 8.13 (iv)]{AshGesztMitrShtTreschl2013}). This belief relied heavily on the assumption that $\Lambda((-1,1)\times \mathbb{R}^{n-1}) = \Lambda(-1,1) = 4\lambda(-1,1)$.

However, in a recent paper \cite{Acampora_Cristoforoni_Nitsch_Trombetti_2026}, we disproved this assumption by showing that $\Lambda((-1,1)\times \mathbb{R}^{n-1}) < \Lambda(-1,1)$, thus opening the door to improvements for the factor $4$ in Payne's inequality. In the present paper, we explicitly compute $\La((-1,1)\times\R^{n-1})$ extending the strategy to the buckling eigenvalue of generic unbounded cylinders.

\section{Proof of the variational characterization}

Before proving the main theorem, we state some properties of the modified bi-harmonic operator $(\De-\mu)^2$ on the space $H^2_0(A)$, where $A$ is a bounded open set in $\R^d$. 

\begin{oss}
We start by observing that, as $(\De-\mu)^2$, on $H^2_0(A)$, is simply a lower-order perturbation of the classical bi-harmonic operator with Dirichlet boundary condition, for every bounded and linear operator $L$ on $H^2_0(A)$, the boundary value problem
\begin{equation}
\label{eqL}
\begin{cases}
    (\De-\mu)^2 u = L &\text{in }A,\\[5pt]
    u=0 &\text{on }\partial A,\\[5pt]
    \partial_\nu u=0 &\text{on }\pa A,
\end{cases}
\end{equation}
admits a unique weak solution $u_L\in H^2_0(A)$, that is the equation
\[\int_A (\De u - \mu u)(\De \varphi - \mu \varphi)\,dx=L(\varphi),\]
for every $\varphi\in H^2_0(A)$, has a unique solution in $H^2_0(A)$. Indeed, for every $\mu\ge0$ the bilinear form
\[\begin{split}B_\mu(u,\vp)&=\int_A (\De u - \mu u)(\De \varphi - \mu \varphi)\,dx\\[7pt]&=\int_A \De u\De \varphi\,dx+2\mu\int_A \nabla u \nabla \varphi\,dx+\mu^2\int_A u\varphi\,dx
\end{split}\]
on $H^2_0(A)$, is equivalent to the standard scalar product of $H^2(A)$
\[\langle u,\vp \rangle _{H^2(A)}=\int_A D^2 u\colon D^2 \varphi\,dx+\int_A \nabla u \nabla \varphi\,dx+\int_A u\varphi\,dx,\]
which, by the Poincaré inequality and integration by parts, is equivalent to 
\[B_0(u,\vp)= \int_A \De u\De \varphi\,dx.\]
Hence, we deduce the existence and uniqueness of a solution by the Lax-Milgram theorem.
\end{oss}

\begin{prop} The boundary value problem
\begin{equation}\label{eq}\begin{cases}
    (\De-\mu)^2 u =-\Lambda (\De-\mu)u &\text{in }A,\\[5pt]
    u=0 &\text{on }\partial A,\\[5pt]
    \partial_\nu u=0 &\text{on }\pa A,
\end{cases}\end{equation}
admits a discrete spectrum 
\[0<\La_{\mu,1}\le\La_{\mu,2}\le\dots\le\La_{\mu,k}\le\dots\to+\infty\]
\end{prop}
\begin{proof}
Endow the space $H^1_0(A)$ with the scalar product
\[b_\mu(u,\vp)=\int_A \nabla u \nabla \varphi\,dx+\mu\int_A u\varphi\,dx.\]
For every $f\in H^1_0(A)$ denote, by $u_f$ the unique solution to equation \eqref{eqL} with $L=L_f$ defined as
\[L_f\colon \vp\in H^2_0(A)\mapsto b_\mu(f,\vp),\]
that is
\begin{equation}\label{B=b}B_\mu(u_f,\vp)=L_f(\vp)=b_\mu(f,\vp),\end{equation}
for every $\vp\in H^2_0(A)$.  In particular
\[B_\mu(u_f,u_f)=b_\mu(f,u_f)\le \max\set{1,\mu}\norma{f}_{H^1_0}\norma{u_f}_{H^1_0},\]
so that
\[\norma{u_f}_{H^2_0}\le C_\mu \norma{f}_{H^1_0},\]
for some constant $C_\mu>0$. Then, by the compactness of the embedding of $H^2_0(A)$ in $H^1_0(A)$, we have that the resolvent map 
\[R_\mu\colon f\in H^1_0(A)\mapsto u_f\in H^2_0(A)\subset H^1_0(A),\]
is a linear and compact operator. Moreover, it is also self-adjoint with respect to the scalar product $b_\mu$, indeed, by definition, we have
\[b_\mu(R_\mu f, g) = B_\mu(R_\mu f, R_\mu g) = b_\mu(f, R_\mu g).\] Hence, by classical spectral theory (see for instance \cite[Theorem 6.11]{B11}), it admits a Hilbert base of eigenfunctions $\Set{f_k}$ with associated eigenvalues $\set{\sigma_k}_k$, that is
\[R_\mu (f_k) = \sigma_k f_k,\]
where the non-zero eigenvalues have finite multiplicity and are either finite in number or can be ordered in a sequence converging to $0$.
Moreover, for every function $f\in H^1_0(A)$
\begin{equation}\label{posdef}b_\mu(R_\mu f, f)=B_\mu(R_\mu f, R_\mu f)\ge0\end{equation}
 and is equal to zero if and only if
$R_{\mu}f=0$. By \eqref{B=b}, if $R_\mu f=0$ we have that $L_f=0$, that is
\[b_\mu(f,\vp)=0\]
for every $\vp\in H^2_0(A)$, hence, by the density of $H^2_0(A)$ in $H^1_0(A)$, $f=0$. In particular, then, $0$ is not an eigenvalue of $R_\mu$.\medskip

Let $\La_{\mu,k}=\si_k^{-1}$ and $u_k=R_\mu(f_k)\in H^2_0(A)$, we have that
\[B_\mu(u_k,\vp)=b_\mu(f_k,\vp)=\La_{\mu,k} b_\mu(u_k,\vp)\]
for every $\vp\in H^2_0(A)$, that is $\set{\La_{\mu,k}}$ and $\set{u_k}$ are the eigenvalues and eigenfunctions of the boundary value problem \eqref{eq}. Finally, by \eqref{posdef}, we have that 
all the eigenvalues $\sigma_k$ are strictly positive, hence the sequence $\set{\La_{\mu,k}}$ can be ordered in an increasing sequence diverging to infinity.
\end{proof}

\begin{oss}
    \label{oss:La>mu}
By the previous remark, denoting by $\La_\mu(A)$ the first eigenvalue, we have the aforementioned variational characterization
\begin{equation}
    \label{eq:varchar}
    \begin{split}
    \La_\mu(A)=&
    \min_{h\in H^2_0(A)\setminus\set{0}} \dfrac{B_\mu(h,h)}{b_\mu(h,h)}
    =
    \min_{h\in H^2_0(A)\setminus\set{0}} \dfrac{\displaystyle\int_A(\De h -\mu h)^2\,dx}{\displaystyle\int_A\abs{\nabla h}^2+\mu h^2\,dx}\\[10pt]
    =&\mu+
    \min_{h\in H^2_0(A)\setminus\set{0}} \dfrac{\displaystyle\int_A(\De h)^2 +\mu \abs{\nabla h}^2\,dx}{\displaystyle\int_A\abs{\nabla h}^2+\mu h^2\,dx},
\end{split}\end{equation}
where we have integrated by parts in the last equality. In particular, we observe that
\[\La_\mu(A)>\mu.\]
\end{oss}

\begin{lemma}\label{lemma:lipschitz-minimum}
The function
\[
    \mu\in[0,+\infty)\mapsto\La_\mu(A)
\]
is locally Lipschitz continuous. Moreover, there exists $\bar\mu\ge0$ such that
\[
    \La_{\bar\mu}(A)=\min_{\mu\ge0}\La_\mu(A).
\]
\end{lemma}

\begin{proof}
For every $h\in H^2_0(A)\setminus\set{0}$, set
\begin{equation}
\label{eq:LGI}
    L(h):=\int_A(\De h)^2\,dx,
    \qquad
    G(h):=\int_A\abs{\nabla h}^2\,dx,
    \qquad
    I(h):=\int_A h^2\,dx.
\end{equation}
By~\eqref{eq:varchar}, the  Rayleigh quotient associated to $\Lambda_\mu(A)$ can be written as
\[
    Q(\mu,h)=\mu+\frac{L(h)+\mu G(h)}{G(h)+\mu I(h)}.
\]
For every $\mu\ge0$, let $h_\mu\in H^2_0(A)$ be an eigenfunction of eigenvalue $\La_\mu(A)$ with $G(h_\mu)+\mu I(h_\mu)=1$, so that
\[
    I(h_\mu)\le \dfrac{G(h_\mu)}{\la(A)},
    \quad G(h_\mu)\le1,\quad\text{and}\quad
L(h_\mu)\le \La_\mu(A),
\]
where we recall that $\la(A)$ denotes the first Dirichlet eigenvalue of the Laplacian.

Let us first notice that $\mu\mapsto\La_\mu(A)$ is locally bounded. For every $M>0$ and $\mu\in[0,M]$, using $h_0$ as a test function for $\La_\mu(A)$, we have 
\begin{equation}
\label{eq:localBoundedness}
\begin{split}\La_\mu(A)&\le Q(\mu,h_0)= \mu+\dfrac{\La_0(A)+\mu}{1+\mu I(h_0)}\\[5pt]&\le 2\mu+\La_0(A)\le 2M+\La_0(A).
\end{split}
\end{equation}

We now show the local Lipschitz continuity. For every $\mu_1,\mu_2\ge0$ and every $h\in H^2_0(A)\setminus\set{0}$, we have that
\[\begin{split}
Q(\mu_1,h)-Q(\mu_2,h) &= \mu_1-\mu_2+ \dfrac{(L(h)+\mu_1 G(h))(G(h)+\mu_2 I(h))-(L(h)+\mu_2 G(h))(G(h)+\mu_1 I(h))}{(G(h)+\mu_2 I(h))(G(h)+\mu_1 I(h))}\\[10pt]
&=(\mu_1-\mu_2)\left(1+ \dfrac{G(h)^2-L(h)I(h)}{(G(h)+\mu_2 I(h))(G(h)+\mu_1 I(h))}\right).\end{split}\]
In particular, for $h=h_{\mu_1}$ and $\mu_1,\mu_2\in[0,M]$, using~\eqref{eq:LGI} and~\eqref{eq:localBoundedness} we have
\[
\begin{split}
   \abs{Q(\mu_1,h_{\mu_1})-Q(\mu_2,h_{\mu_1})} &\le\left(1+\dfrac{G(h_{\mu_1})^2+L(h_{\mu_1})I(h_{\mu_1})}{G(h_{\mu_1})+\mu_2 I(h_{\mu_1})}\right)\abs{\mu_1-\mu_2}\\[10pt]
   &\le \left(2+\dfrac{\La_\mu(A)}{\la(A)}\right)\abs{\mu_1-\mu_2}\\[10pt]
   &\le \left(2+\dfrac{2M+\La_0(A)}{\la(A)}\right)\abs{\mu_1-\mu_2}=:L_M \abs{\mu_1-\mu_2}.
\end{split}
\]
Then
\[\La_{\mu_2}(A)\le Q(\mu_2,h_{\mu_1})\le \La_{\mu_1}(A)+\abs{Q(\mu_1,h_{\mu_1})-Q(\mu_2,h_{\mu_1})}\le \La_{\mu_1}(A)+L_M \abs{\mu_1-\mu_2}.\]
Similarly, interchanging the roles of $\mu_1$ and $\mu_2$, we obtain
\[\La_{\mu_1}(A)\le\La_{\mu_2}(A)+L_M\abs{\mu_1-\mu_2},\]
hence
\[\abs{\La_{\mu_1}(A)-\La_{\mu_2}(A)}\le L_M\abs{\mu_1-\mu_2}.\]

Finally, by \autoref{oss:La>mu}, we have $\La_\mu(A)>\mu$ for every $\mu\ge0$, which, together with the continuity of $\La_\mu(A)$, proves the existence of a minimum point $\bar \mu$.
\end{proof}

\begin{oss}
Notice that $\La_\mu(A)$ enjoys the following scaling property 
\[\La_{t^2\mu}(tA)=t^{-2}\La_\mu(A).\]
\end{oss}
\medskip

We now prove the main theorem.
\begin{proof}[Proof of \autoref{main}]
    We start by proving that $\Lambda(A\times\R^{n-d})\le \Lambda_\mu(A)$ for every $\mu\ge0$. Let $k=n-d$, fix $\chi\in C^\infty_c(\R^k)$ with 
    \[\int_{\R^{k}}\chi^2\,dx=1.\]
    For every $\mu\ge0$, every $\xi\in \R^{k}$, with $\abs{\xi}^2=\mu$, and $R>0$ let us define the functions $u_{S}, u_{C}\in C^\infty_c(\R^k)$
    \[
    u_{S}(x)=R^{-k/2}\chi\left(\dfrac{x}{R}\right)\sin(\xi\cdot x),\qquad u_{C}(x)=R^{-k/2}\chi\left(\dfrac{x}{R}\right)\cos(\xi\cdot x).
    \]
    Let $h\in C^\infty_c(A)\setminus\set{0}$  and consider
    \[
    v_{S}(x,y)=u_{S}(x)h(y),\qquad v_{C}(x,y)=u_{C}(x)h(y).
    \]
     Using $v_{S}$ and $v_{C}$ as test functions for $\Lambda(A\times\R^k)$,  we have  
    \[
    \begin{split}
        \Lambda(A\times\R^k)&\le \min\Set{\dfrac{\displaystyle\int_{A\times\R^k} (\Delta v_{S})^2\,dx\,dy}{\displaystyle\int_{A\times\R^k} \abs{\nabla v_{S}}^2\,dx\,dy },\dfrac{\displaystyle\int_{A\times\R^k} (\Delta v_{C})^2\,dx\,dy}{\displaystyle\int_{A\times\R^k} \abs{\nabla v_{C}}^2\,dx\,dy }} \\[7 pt]
        &\le \dfrac{\displaystyle\int_{A\times\R^k} \left(\Delta v_{S}\right)^2+\left(\Delta v_{C}\right)^2\,dx\,dy}{\displaystyle\int_{A\times\R^k}  \abs{\nabla v_{S}}^2+\abs{\nabla v_{C}}^2\,dx\,dy }.
    \end{split}
        \]
    By direct computations, using that 
    \[
        u_S^2+u_C^2 = R^{-k}\chi^2(x/R), \qquad   |\na u_S|^2+|\na u_C|^2 =\mu R^{-k}\chi^2(x/R)+ O(R^{-k-2}),
    \]
    we have 
    \begin{equation}\label{eqgrad}
    \begin{split}
        \abs{\nabla v_{S}}^2+\abs{\nabla v_{C}}^2 = R^{-k} \chi^2({x/R}) \left(\abs{\nabla h}^2+\mu h^2\right)+ O (R^{-k-2}).
    \end{split}
    \end{equation}
    For the Laplacian, we observe that 
    \[
        \De u_S = -\mu u_S + O(R^{-k/2-1}), \qquad \De u_C = -\mu u_C + O(R^{-k/2-1}),
    \]
    so that 
    \begin{equation}\label{eqlap}\begin{split}\left(\Delta v_{S}\right)^2+\left(\Delta v_{C}\right)^2=&R^{-k}\chi^2|_{x/R} \left(\Delta h - \mu h\right)^2+O(R^{-k-2}).
    \end{split}\end{equation}
    Integrating \eqref{eqgrad} and \eqref{eqlap}, and using the fact that $R^{-k/2}\chi$ is normalized in $L^2(\R^k)$ and that $\chi(x/R)$ is supported in $R\supp(\chi)$, we have that 
        \[
        \begin{gathered}\int_{A\times\R^k}\left(\left(\Delta v_{S}\right)^2+\left(\Delta v_{C}\right)^2\right)\,dx\,dy=\int_A \left(\Delta h - \mu h\right)^2\,dy+O(R^{-2}),\\[15pt]
        \int_{A\times\R^k}\left(\abs{\nabla v_{S}}^2+\abs{\nabla v_{C}}^2\right)\,dx\,dy=\int_A \left(\abs{\nabla h}^2+\mu h^2\right)\,dy+O(R^{-2}),
        \end{gathered}
        \]
        so that 
        \[\Lambda(A\times\R^k)\le \dfrac{\displaystyle \int_A \left(\Delta h - \mu h\right)^2\,dy}{\displaystyle \int_A \abs{\nabla h}^2+\mu h^2\,dy}+o(1),\]
        and taking the limit as $R$ goes to infinity and the infimum over $h\in C^\infty_c(A)\setminus\set{0}$ we have $\Lambda(A\times\R^{n-d})\le \Lambda_\mu(A)$ for every $\mu\ge0$.\medskip

        To prove the reverse inequality, fix $v \in C_c^\infty(A\times\R^k)\setminus\{0\}$, and let $\widehat v(\xi,y)$ be the Fourier transform of $v$ with respect to the component $x\in\R^k$. By Plancherel,
\[
\int_{A\times\R^k} (\Delta v)^2\,dx\,dy
=
\int_{\R^k} \int_A \left(\Delta_y \widehat{v}(\xi,y)-\abs{\xi}^2\widehat{v}(\xi,y)\right)^2\,dy\,d\xi
\]
and
\[
\int_{A\times\R^k} |\nabla v|^2\,dx\,dy
=
\int_{\R^k} \int_A \left(\abs{\nabla_y \widehat{v}(\xi,y)}^2+\abs{\xi}^2\abs{\widehat{v}(\xi,y)}^2\right) \,dy\,d\xi.
\]
By definition of $\Lambda_\mu(A)$, for every $\xi\in \R^k$ we have 
\[\begin{split}
 \int_A \left(\Delta_y \widehat{v}(\xi,y)-\abs{\xi}^2\widehat{v}(\xi,y)\right)^2\,dy&\ge \Lambda_{\abs{\xi}^2}(A) \int_A \left(\abs{\nabla_y \widehat{v}(\xi,y)}^2+\abs{\xi}^2\abs{\widehat{v}(\xi,y)}^2\right)\,dy\\[15pt]
 &\ge \inf_{\mu\ge0} \Lambda_{\mu}(A) \int_A \left(\abs{\nabla_y \widehat{v}(\xi,y)}^2+\abs{\xi}^2\abs{\widehat{v}(\xi,y)}^2\right)\,dy.
\end{split}\]
Integrating in $\xi$, we obtain
\[
\int_{A\times\R^k} (\Delta v)^2\,dx\,dy
\geq
\inf_{\mu\ge0}\Lambda_\mu(A)
\int_{A\times \R^k} \abs{\nabla v}^2\,dx\,dy.
\]
Therefore, for every $v \in C_c^\infty(A\times\R^k)\setminus\set{0}$
\[
\frac{\displaystyle \int_{A\times\R^k} (\Delta v)^2\,dx\,dy}{\displaystyle \int_{A\times\R^k} \abs{\nabla v}^2\,dx\,dy} \geq \inf_{\mu\ge0}\Lambda_\mu(A)
\]
and taking the infimum over $v$ yields
\[
\Lambda(A\times\R^k) \geq \inf_{\mu\ge0}\Lambda_\mu(A).
\]
\end{proof}

\subsection{Explicit computations on balls}

In this section, we compute the eigenvalues and eigenfunctions of the problem
\[\begin{cases}
        (\De-\mu)^2 u=-\Lambda (\De-\mu)u &\text{in }B_R,\\[5pt]
        u=0 &\text{on }\pa B_R,\\[5pt]
        \pa_\nu u=0 &\text{on }\pa B_R,
    \end{cases}\]
of a ball $B_R\subset\R^d$, in terms of the Bessel functions, the modified Bessel functions of the first kind and of the spherical harmonics. In particular, we first discuss the case $d>1$ and $\mu>0$ in \autoref{prop:spectrum}, and $d=1$ and $\mu>0$ in \autoref{prop:spectrum1D}. Finally in \autoref{oss:mu=0} we recall the case of the Buckling problem i.e. the case $\mu=0$. We conclude the section with the study of the minimization of the function
\[\mu\mapsto\La_\mu(B_1),\]
in \autoref{minim}, yielding the proof of \autoref{main2}. Similar computations for the eigenvalues and eigenfunctions of fourth-order PDEs in radially symmetric domains can be found, for instance, in \cite{Coffman_Duffin_1992, FGW05, C11, BP15, DCNT15, DCNT17, BP20,Buoso_Luzzini_Provenzano_Stubbe_2021,BF25,BM25}.\medskip

 For every $\nu\in\R$ let $J_\nu$ and $I_\nu$ be the Bessel function and modified Bessel function of the first kind of order $\nu$ respectively. For every $\nu\ge0$ and every $i\in\N$ we denote by $j_{\nu,i}$ and $j'_{\nu,i}$ the $i$-th positive zero of $J_\nu$ and $J'_{\nu}$ respectively. We recall the interlacing property of the zeros of the Bessel functions of positive order (see \cite[Formulae 9.5.2]{ASHandbook})
\begin{equation}
    \label{eq:interlacing}
    \begin{split}
&j_{\nu,1}<j_{\nu+1,1}<j_{\nu,2}<j_{\nu+1,2}<\dots\\[5pt]
&\nu\le j'_{\nu,1}<j_{\nu,1}<j'_{\nu,2}<j_{\nu,2}<\dots
\end{split}
\end{equation}
To avoid confusion with the modified Bessel of the second kind, for every $k\in\N_0$ let $\Set{S_{k,l}}_{1\le l\le d_k}$ be an $L^2(\partial B_1)$-orthonormal basis for the spherical harmonics of order $k$, that is an orthonormal basis of the eigenfunctions of the Laplace-Beltrami operator on $\partial B_1$ of eigenvalue $\lambda_k=k(d-2+k)$. 
\medskip

We start with the case $d>1$.

\begin{prop}\label{prop:spectrum}
    Let $d>1$, let $B_R\subseteq\R^d$ be a centered ball, and fix $\mu>0$. Then, denoting by $\eta=d/2-1$ and $\eta_k=\eta+k$, we have that the spectrum of the eigenvalue problem 
    \begin{equation}\label{eqballu}\begin{cases}
        (\De-\mu)^2 u=-\Lambda (\De-\mu)u &\text{in }B_R,\\[5pt]
        u=0 &\text{on }\pa B_R,\\[5pt]
        \pa_\nu u=0 &\text{on }\pa B_R,
    \end{cases}\end{equation}
    is given by the solutions to the following family of equations
    \begin{equation}\label{spectrum}\sqrt{\La -\mu}\,I_{\eta_k}(\sqrt\mu\, R) J_{\eta_k+1}\left(\sqrt{\La-\mu}\, R\right)+\sqrt\mu\,I_{\eta_k+1}(\sqrt\mu\, R) J_{\eta_k}\left(\sqrt{\La-\mu}\, R\right)=0,\end{equation}
    each of which has infinitely many solutions. Denoting by $\La_{\mu,k,i}$ the $i$-th zero of the previous equation we have that
    \[\dfrac{j _{\eta_k,i}^2}{R^2}+\mu<\La_{\mu,k,i}<\dfrac{j _{\eta_k+1,i}^2}{R^2}+\mu.\]
    Moreover, for every $1\le l \le d_k$, the function 
    \[u_{k,l,i}(x)=S_{k,l}\left(\dfrac{x}{\abs{x}}\right)\abs{x}^{-\eta}\left[\dfrac{I_{\eta_k}(\sqrt\mu\, \abs{x})}{I_{\eta_k}(\sqrt{\mu}\,R)}-\dfrac{J_{\eta_k}\left(\sqrt{\La_{\mu,k,i}-\mu}\,\abs{x}\right)}{J_{\eta_k}\left(\sqrt{\La_{\mu,k,i}-\mu}\,R\right)}\right],\]
    is an associated eigenfunction.
    In particular, the smallest eigenvalue, $\Lambda_\mu$, is simple and $\La_\mu=\Lambda_{\mu,0,1}$.
\end{prop}
\begin{oss}
    We remark that by the interlacing property \eqref{eq:interlacing}, $\La_{\mu,k,i}$ is monotone in $k$ and in $i$; however, we are not able to infer the injectivity of the map $(k,i)\mapsto\La_{\mu,k,i}$.
\end{oss}
\begin{proof}
     Let $u$ be a weak solution on $B_R$ to \eqref{eq}. Letting $v=(\De-\mu)u$, we have that $v$ satisfies
    \[\int_{B_R} \nabla v\nabla \varphi\,dx=\left(\Lambda-\mu\right)\int_{B_R}v\varphi\,dx,\]
    for every $\varphi\in C^\infty_c(B_R)$, so that $v$ is $C^\infty$ inside $B_R$ and so is $u$. In particular,
    \[(\Delta-\mu)^2u=-\La(\Delta-\mu)u \quad\text{in }B_R,\]
    hence
    \begin{equation}\label{veq}-\De v = (\La-\mu)v\quad\text{in }B_R.\end{equation}
    As $\vp_0=-v/\La$ is a particular solution to $(\De-\mu)\vp_0 = v$, we have that $u$, which is a solution to
    \[\begin{cases}
    (\Delta-\mu)u=v &\text{in }B_R,\\[5pt]
    u=\partial_\nu u=0 &\text{on }\pa B_R,
    \end{cases}\]
    can be written as $u=w_0-v/\La$, where $w_0$ is a solution to
    \begin{equation}\label{uoeq}-\De w_0=-\mu w_0\quad\text{in }B_R.\end{equation}
    Expanding in Fourier series on every sphere, we can write 
    \[v(x)=\sum_{k=0}^{+\infty}\sum_{l=1}^{d_k} S_{k,l}\left(\dfrac{x}{\abs{x}}\right)a_{k,l}(|x|),
    \qquad \qquad w_0(x)=\sum_{k=0}^{+\infty}\sum_{l=1}^{d_k} S_{k,l}\left(\dfrac{x}{\abs{x}}\right)b_{k,l}(|x|).\]
    Since 
    \[
    a_{k,l}(r)=r^{1-d}\int_{\pa B_r}v(x)S_{k,l}\left(\frac{x}{|x|}\right)\,d\mathcal{H}^{d-1},
    \]
    using that $S_{k,l}$ is an eigenfunction for the Laplace-Beltrami operator on the sphere $\pa B_1$, and integrating by parts, we have that
    \[
    \begin{split}
        -\frac{\la_k}{r^2} a_{k,l}(r)&=r^{1-d}\int_{\pa B_r}(\De_{\pa B_r} v)\:S_{k,l}\,d\mathcal{H}^{d-1} = r^{1-d}\int_{\pa B_r}\left(\De v-\pa^2_{rr}v-\frac{d-1}{r}\pa_r v\right)\:S_{k,l}\,d\mathcal{H}^{d-1} \\ 
        &=(\mu-\La)a_{k,l}(r)-a''_{k,l}(r)-\frac{d-1}{r}a'_{k,l}(r),
    \end{split}
    \]
    where we used the equation~\eqref{veq} for $v$. Reasoning similarly for $w_0$, we get that
    \[
        -\frac{\la_k}{r^2}b_{k,l}(r)=\mu b_{k,l}(r)-b''_{k,l}(r)-\frac{d-1}{r}b'_{k,l}(r).
    \]
    By definition of $a_{k,l}$ and $b_{k,l}$, and using the regularity of $v,w_0$, we know that $a_{k,l}$ and $b_{k,l}$  they are bounded in $0$. Therefore, $a_{k,l}$ and $b_{k,l}$ can be written in terms of Bessel functions and modified Bessel functions of the first kinds, so that there exist constants $\ga_{k,l}$ and $\al_{k,l}$ such that 
    \[
a_{k,l}(r)=r^{-\eta}\ga_{k,l}J_{\eta_k}(\sqrt{\La-\mu}\,r)\qquad \qquad b_{k,l}(r)=r^{-\eta}\al_{k,l} I_{\eta_k}(\sqrt{\mu}\,r),\]
where we recall that $\eta=d/2-1$ and $\eta_k=\eta+k$. Therefore, $v$ and $w_0$ are of the form
    \[v(x)=\abs{x}^{-\eta}\sum_{k=0}^{+\infty}\sum_{l=1}^{d_k} \ga_{k,l}S_{k,l}\left(\dfrac{x}{\abs{x}}\right) J_{\eta_k}(\sqrt{\La-\mu}\,\abs{x})\]
    and
    \[w_0(x)=\abs{x}^{-\eta}\sum_{k=0}^{+\infty}\sum_{l=1}^{d_k} \al_{k,l}S_{k,l}\left(\dfrac{x}{\abs{x}}\right) I_{\eta_k}(\sqrt{\mu}\,\abs{x}).\]
    Setting $\beta_{k,l}=-\La^{-1}\ga_{kl}$ and recalling that $u=w_0-\La^{-1}v$, we have
    \[u(x)=\sum_{k=0}^{+\infty}\sum_{l=1}^{d_k} u_{k,l},\]
    with
    \[u_{k,l}(x)=\abs{x}^{-\eta} S_{k,l}\left(\dfrac{x}{\abs{x}}\right)\left[\al_{k,l} I_{\eta_k}(\sqrt{\mu}\,\abs{x})+\beta_{k,l}J_{\eta_k}(\sqrt{\La-\mu}\,\abs{x})\right].\]
Let us notice that, by construction, all the functions $u_{k,l}$ solve the equation $(\De-\mu)^2 u_{k,l}=-\La(\De-\mu)u_{k,l}$. Imposing the boundary conditions on the eigenfunction $u$, and using the orthogonality of the system $\Set{S_{k,l}}_{k,l}$, we have that $u$ is a non-trivial solution to the boundary value problem~\eqref{eqballu} if and only if at least one of the functions $u_{k,l}$ is non-zero and it satisfies the boundary conditions. Hence, using the differentiation formulae (see, for instance, \cite[Equations 9.1.30 and 9.6.28]{ASHandbook})
\begin{equation}
    \label{eq:derivativeBessel1}
\frac{d}{dz}\bigl(z^{-\nu}I_{\nu}(z)\bigr)=z^{-\nu}I_{\nu+1}(z),
\qquad
\frac{d}{dz}\bigl(z^{-\nu}J_{\nu}(z)\bigr)=-z^{-\nu}J_{\nu+1}(z),
\end{equation}
and imposing the boundary conditions on $u_{k,l}$, we have that $u_{k,l}$ is an eigenfunction if and only if the system 
\begin{equation}\label{system}\begin{cases}
     \al\,I_{\eta_k}(\sqrt\mu\,R)+\be\,J_{\eta_k}\left(\sqrt{\Lambda-\mu}\,R\right)=0,\\[0.7ex]
\al\left[\,\sqrt\mu\,R\,I_{\eta_k+1}(\sqrt\mu\,R)+kI_{\eta_k}(\sqrt\mu\,R)\right]
+\be\left[kJ_{\eta_k}(\sqrt{\La-\mu}\,R)-\sqrt{\Lambda-\mu}\,R\,J_{\eta_k+1}\left(\sqrt{\Lambda-\mu}\,R\right)\right] = 0
\end{cases}\end{equation}
admits a non-trivial solution $(\al,\be)$, that is, if and only if $\La$ is a solution to~\eqref{spectrum}.

To prove the existence of a solution to~\eqref{spectrum}, it suffices to notice that for every $k\in \N_0$ the function 
 \begin{equation*}
    F(t) = \frac{t}{R}\,I_{\eta_k}(\sqrt\mu\, R) J_{\eta_k+1}\left(t\right)+\sqrt\mu\,I_{\eta_k+1}(\sqrt\mu\, R) J_{\eta_k}\left(t\right),
 \end{equation*}
 by the interlacing property~\eqref{eq:interlacing} and the positivity of $I_\nu$, satisfies $F(j_{\eta_k,i})F(j_{\eta_k+1,i})<0$ for every $i\in \N$. 
 
 Therefore, the equation~\eqref{spectrum} admits at least one zero for $\sqrt{\La-\mu}\,R$ in the interval $(j_{\eta_k,i},j_{\eta_k+1,i})$. Moreover, recalling the recursive relations of Bessel functions (see \cite[formulae 9.1.27]{ASHandbook})
 \[J'_\nu(y)=\dfrac{\nu}{y}J_\nu(y)-J_{\nu+1}(y),\quad\text{and}\quad J'_{\nu+1}(y)=J_\nu(y)-\dfrac{\nu+1}{y}J_{\nu+1}(y);\]
 we also have that, for $t\in(j_{\eta_k,i},j_{\eta_k+1,i})$ , the functions
 \[t\mapsto t\,J_{\eta_k+1}(t),\quad\text{and}\quad t\mapsto J_{\eta_k}(t),\]
 have the same monotonicity, so that
 \begin{equation}
 \label{eq:F'>0}
     F'(t)>0, \qquad t\in(j_{\eta_k,i},j_{\eta_k+1,i}),
 \end{equation}
 and in particular, for $\sqrt{\La-\mu}\,R\in(j_{\eta_k,i},j_{\eta_k+1,i})$, the function $F(\sqrt{\La-\mu}R)$ is a monotone function of $\La$. Hence, for every $k\in\N_0$, equation \eqref{spectrum} admits exactly one zero for $\sqrt{\La-\mu}\,R$ in the interval $(j_{\eta_k,i},j_{\eta_k+1,i})$, that is,  the $i$-th zero of~\eqref{spectrum} must satisfy 
\[j_{\eta_k,i}<\sqrt{\La_{\mu,k,i}-\mu} \:R<j_{\eta_k+1,i}.\]
In particular, we have $J_{\eta_k}\left(\sqrt{\Lambda_{\mu,k,i}-\mu}\,R\right)\ne 0$, so that 
\[\al=I_{\eta_k}(\sqrt\mu\,R)^{-1},\quad\text{and}\quad\be=-J_{\eta_k}\left(\sqrt{\Lambda-\mu}\,R\right)^{-1}\]
is a non-trivial solution to~\eqref{system}. We then get 
\[u_{k,l,i}(x)=S_{k,l}\left(\dfrac{x}{\abs{x}}\right)\abs{x}^{-\eta}\left[\dfrac{I_{\eta_k}(\sqrt\mu\, \abs{x})}{I_{\eta_k}(\sqrt{\mu}\,R)}-\dfrac{J_{\eta_k}\left(\sqrt{\La_{\mu,k,i}-\mu}\,\abs{x}\right)}{J_{\eta_k}\left(\sqrt{\La_{\mu,k,i}-\mu}\,R\right)}\right]\]
is an eigenfunction of eigenvalue $\La_{\mu,k,i}$. Finally, as $j_{\eta+1,1}\le j_{\eta_k,i}$ for every $(k,i)\ne(0,1)$, we have that the smallest eigenvalue of~\eqref{eqballu}, $\Lambda_\mu$, is $\La_{\mu,0,1}$ which in particular is simple, the corresponding eigenfunction is radial and proportional to 
\begin{equation}
    \label{eq:firstEigenfunctionBall}
    u(x)=\abs{x}^{-\eta}\left[\dfrac{I_{\eta}(\sqrt\mu\, \abs{x})}{I_{\eta}(\sqrt{\mu}\,R)}-\dfrac{J_{\eta}\left(\sqrt{\La_{\mu}-\mu}\,\abs{x}\right)}{J_{\eta}\left(\sqrt{\La_{\mu}-\mu}\,R\right)}\right].\end{equation}
\end{proof}

\begin{oss}
Let us note that the first eigenfunction $u$ in~\eqref{eq:firstEigenfunctionBall} is positive and radially decreasing. Before proving it, let us stress that for fourth-order operators the sign of the first eigenfunction is strictly related to the geometry of the set, as the Krein-Rutman theorem cannot hold in general (we refer to~\cite{Coffman_1982,Kozlov_Kondratev_Mazya_1990,Coffman_Duffin_1992,Laurencot_Walker_2015} for the clamped plate eigenvalue and to~\cite{Kozlov_Kondratev_Mazya_1990,DCNT17,BP20} for the buckling; see also~\cite[Theorems 3.7-3.9]{Gazzola_Grunau_Sweers_2010}).

For simplicity, let us denote by
\[
    a:=\sqrt\mu,
    \qquad
    b:=\sqrt{\La_\mu-\mu}.
\]
With a slight abuse of notation, we identify the function $u$ with its radial profile, that is, $u(x)=u(\abs{x})$. We differentiate $u(r)$ with respect to $r$ using again~\eqref{eq:derivativeBessel1}, so that,
\[
    u'(r)=r^{-\eta}\left[
    a\frac{I_{\eta+1}(ar)}{I_\eta(aR)}
    +b\frac{J_{\eta+1}(br)}{J_\eta(bR)}
    \right].
\]
The equation~\eqref{spectrum} characterizing the first eigenvalue reads
\[
    b\frac{I_\eta(aR)}{J_\eta(bR)}=-a\frac{I_{\eta+1}(aR)}{J_{\eta+1}(bR)},
\]
and in particular we can rewrite $u'(r)$ as
\[
    u'(r)=\frac{a r^{-\eta}J_{\eta+1}(br)}{I_\eta(aR)}
    \left[
    \frac{I_{\eta+1}(ar)}{J_{\eta+1}(br)}
    -
    \frac{I_{\eta+1}(aR)}{J_{\eta+1}(bR)}
    \right].
\]
Since $bR<j_{\eta+1,1}$, we have $J_{\eta+1}(br)>0$ for every $r\in(0,R]$. Therefore, to prove the monotonicity of $u$ it is sufficient to show that the map
\[
    r\mapsto \frac{I_{\eta+1}(ar)}{J_{\eta+1}(br)}
\]
is strictly increasing on $(0,R)$.

We prove this monotonicity using the product formulae (see~\cite[Equations 9.5.10, and 9.6.3]{ASHandbook}). Letting $\nu=\eta+1$, we have
\[
    I_\nu(z)=\frac{(z/2)^\nu}{\Gamma(\nu+1)}
    \prod_{m=1}^\infty\left(1+\frac{z^2}{j_{\nu,m}^2}\right),
\]
and
\[
    J_\nu(z)=\frac{(z/2)^\nu}{\Gamma(\nu+1)}
    \prod_{m=1}^\infty\left(1-\frac{z^2}{j_{\nu,m}^2}\right),
\]
so that
\[
    \frac{I_\nu(ar)}{J_\nu(br)}
    =\left(\frac{a}{b}\right)^\nu
    \prod_{m=1}^\infty
    \frac{j_{\nu,m}^2+a^2r^2}{j_{\nu,m}^2-b^2r^2}.
\]
As $bR<j_{\nu,1}$, every factor is positive for $0<r\le R$. Moreover, recalling that  $a^2+b^2=\La_\mu>0$, we have that, for every $m$, the map
\[
    r\longmapsto
    \frac{j_{\nu,m}^2+a^2r^2}{j_{\nu,m}^2-b^2r^2} = 1+\frac{a^2+b^2}{b^2}\ml(\frac{j_{\nu,m}^2}{j_{\nu,m}^2-b^2r^2}-1\mr)
\]
is strictly increasing on $(0,R)$.
Hence $r\mapsto I_\nu(ar)/J_\nu(br)$ is strictly increasing on $(0,R)$, and consequently $u'(r)<0$ for every $r\in(0,R)$. Using the boundary condition $u(R)=0$, also gives $u(r)>0$ for $0<r<R$.
\end{oss}

We now turn to the $1$-dimensional case. 
\begin{prop}
\label{prop:spectrum1D}
    Let $d=1$, and fix $\mu>0$. Then the problem
    \[\begin{cases}
   -(u''-\mu u)''=(\La-\mu) (u''-\mu u)&\text{in }(-R,R),\\[5pt]
    u(\pm R)=0,\\[5pt]
    u'(\pm R)=0,
    \end{cases}\]
    admits $\La$ as an eigenvalue if and only if it solves either
    \begin{equation}\label{spectrum1d1}
    \sqrt{\La-\mu}\,\tan(\sqrt{\La-\mu}\,R)=-\sqrt{\mu}\,\tanh(\sqrt\mu\,R),
    \end{equation}
    or
        \begin{equation}\label{spectrum1d2}
\dfrac{\tan(\sqrt{\La-\mu}\,R)}{\sqrt{\La-\mu}}=\dfrac {\tanh(\sqrt\mu\,R)}{\sqrt\mu}.
    \end{equation}
    Moreover, denoting by $\La_{\mu,0,i}$ and $\La_{\mu,1,i}$ the $i$-th zero of equations \eqref{spectrum1d1} and \eqref{spectrum1d2} respectively, we have 
    \[\left(i-\dfrac{1}{2}\right)^2\dfrac{\pi^2}{R^2}+\mu<\La_{\mu,0,i}<\dfrac{i^2\pi^2}{R^2}+\mu,\quad\text{and}\quad\dfrac{i^2\pi^2}{R^2}+\mu<\La_{\mu,1,i}<\left(i+\dfrac{1}{2}\right)^2\dfrac{\pi^2}{R^2}+\mu \]
    with corresponding eigenfunctions
    \[u_{0,i}(x)=\dfrac{\cosh(\sqrt\mu\, x)}{\cosh(\sqrt\mu\, R)}-\dfrac{\cos\left(\sqrt{\La_{\mu,0,i}-\mu}\, x\right)}{\cos\left(\sqrt{\La_{\mu,0,i}-\mu}\, R\right)},\]
    and
    \[u_{1,i}(x)=\dfrac{\sinh(\sqrt\mu\, x)}{\sinh(\sqrt\mu\, R)}-\dfrac{\sin\left(\sqrt{\La_{\mu,1,i}-\mu}\, x\right)}{\sin\left(\sqrt{\La_{\mu,1,i}-\mu}\, R\right)}.\]
    In particular, the first eigenvalue $\La_\mu$ is given by $\La_{\mu,0,1}$.
\end{prop}
\begin{proof}
    The general solution to the fourth-order equation
   \[-(u''-\mu u)''=(\La_\mu-\mu) (u''-\mu u), \]
is 
\[u(x)=\al_1\cos\left(\sqrt{\La-\mu}\,x\right)+\beta_1\cosh(\sqrt{\mu}\,x)+\al_2\sin\left(\sqrt{\La-\mu}\,x\right)+\beta_2\sinh(\sqrt{\mu}\,x).\]
By imposing the boundary conditions $u(R)+u(-R)=u(R)-u(-R)=0$ and $u'(R)+u'(-R)=u'(R)-u'(-R)=0$, we have the system
\[\begin{cases}
    \al_1\cos(\sqrt{\La-\mu}\,R)+\be_1\cosh(\sqrt\mu\,R)=0,\\[5pt]
    -\al_1\sqrt{\La-\mu}\,\sin(\sqrt{\La-\mu}\,R)+\be_1\sqrt\mu\,\sinh(\sqrt\mu\,R)=0,\\[5pt]
    \al_2\sin(\sqrt{\La-\mu}\,R)+\be_2\sinh(\sqrt\mu\,R)=0,\\[5pt]
    \al_2\sqrt{\La-\mu}\,\cos(\sqrt{\La-\mu}\,R)+\be_2\sqrt{\mu}\,\cosh(\sqrt\mu\,R)=0,
\end{cases}\]
Which admits non-trivial solutions if and only if
\[\sqrt{\mu}\,\sinh(\sqrt\mu\,R)\cos(\sqrt{\La-\mu}\,R)+\sqrt{\La-\mu}\,\cosh(\sqrt\mu\,R)\sin(\sqrt{\La-\mu}\,R)=0,
\]
or
\[
\sqrt{\mu}\,\cosh(\sqrt\mu\,R)\sin(\sqrt{\La-\mu}\,R)-\sqrt{\La-\mu}\,\sinh(\sqrt\mu\,R)\cos(\sqrt{\La-\mu}\,R)=0.
\]
which are exactly \eqref{spectrum1d1} and \eqref{spectrum1d2}.
Finally, the estimates on the eigenvalues follow trivially by considering the positivity intervals and monotonicity of the functions
\[t\mapsto t\tan(t),\quad\text{and }\quad t\mapsto\dfrac{\tan(t)}{t},\]
and by the fact that, for $0<\sqrt{\La-\mu}\,R<\pi/2$,  \eqref{spectrum1d2} does not admit solutions. 
\end{proof}
\begin{oss}\label{oss:mu=0}
    For every dimension $d$, the case $\mu=0$ reduces to the usual buckling eigenvalue problem. Even though we need $\mu>0$ in both the proofs of~\autoref{prop:spectrum} and~\autoref{prop:spectrum1D}, we can still argue similarly to compute the spectrum. Indeed, the main difference here is that the function $w_0$ is harmonic. In the case $d>1$, we replace $I_{\eta_k}(\sqrt{\mu}x)$ with $|x|^{\eta_k}=|x|^{\eta+k}$, so that
    \[w_0(x)=\sum_{k=0}^{+\infty}\sum_{l=1}^{d_k}\al_{k,l}S_{k,l}\left(\dfrac{x}{\abs{x}}\right)\abs{x}^{k},\]
    and
    \[u_{k,l}(x)=S_{k,l}\left(\dfrac{x}{\abs{x}}\right)\left[\al_{k,l} \abs{x}^{k}+\beta_{k,l}\abs{x}^{-\eta} J_{\eta_k}(\sqrt{\La}\,\abs{x})\right].\]
Then condition \eqref{spectrum} reduces to 
\begin{equation}\label{spectrum0}
    J_{\eta_k+1}(\sqrt\La\,R)=0,\end{equation}
that is $\La_{0,k,i}=j_{\eta_k+1,i}^2R^{-2}$, with eigenfunctions 
\[u_{k,l,i}(x)=\abs{x}^{-\eta} S_{k,l}\left(\dfrac{x}{\abs{x}}\right)\left[\left(\dfrac{\abs{x}}{R}\right)^{\eta_k}-\dfrac{J_{\eta_k}(j_{\eta_k+1,i} R^{-1}\,\abs{x})}{J_{\eta_k}(j_{\eta_k+1,i})}\right].\]
In the case $d=1$ we replace $\cosh(\sqrt{\mu}x)$ and $\sinh(\sqrt{\mu}x)$ with $1$ and $x$ respectively, and the general solution is given by
\[u(x)=\al_1\cos\left(\sqrt{\La}\,x\right)+\beta_1+\al_2\sin\left(\sqrt{\La}\,x\right)+\beta_2x.\]
Then, the eigenvalues are given by the equations
\[\La_{0,0,i}=\dfrac{i^2\pi^2}{R^2},\]
for $i\in\N$, with relative eigenfunctions
\[u_{0,i}(x)=1-\dfrac{\cos(\sqrt\La\, x)}{\cos(\sqrt\La\, R)};\] for the eigenvalues $\La_{0,1,i}$, instead, we look for the solutions to the analogous of~\eqref{spectrum1d2}, which is 
\[\tan\left(\sqrt{\La_{0,1,i}}\,R\right)=\sqrt{\La_{0,1,i}}\,R,\]
and the relative eigenfunctions are
\[u_{1,i}(x)=\dfrac{x}{R}-\dfrac{\sin\left(\sqrt{\La_{0,1,i}}\, x\right)}{\sin\left(\sqrt{\La_{0,1,i}}\, R\right)}.\]
\end{oss}

We conclude the section with the study of the minimization problem
\[\min_{\mu\ge0}\La_\mu(B_1).\]
\begin{prop}\label{minim}
    Let $\La_\mu(B_1)$ be the first eigenvalue of \eqref{spectrum} in $B_1\subset\R^d$, and let $\bar\mu$ be a minimizer of $\La_\mu$. Then either $\bar\mu=0$ or $d=1$ and $\bar\mu$ is the unique solution to the equation
   \begin{equation}
   \label{eq:optimalMu}
  \sqrt\mu\,\tanh(\sqrt\mu)=1.
   \end{equation}
\end{prop}
\begin{proof}
Let $d>1$, and let
\[
G(\Lambda,\mu)=\sqrt{\Lambda-\mu}\,I_{\eta}(\sqrt\mu)J_{\eta+1}(\sqrt{\Lambda-\mu})
+ \sqrt\mu\,J_{\eta}(\sqrt{\Lambda-\mu})I_{\eta+1}(\sqrt\mu),
\]
so that equation~\eqref{spectrum} for the first eigenvalue reads $G(\La_\mu,\mu)=0$. 
As discussed in \autoref{prop:spectrum} (see~\eqref{eq:F'>0} with $R=1$), the function $G(\La,\mu)$ is strictly monotone in $\La$ in the interval $(j_{\eta,1}^2+\mu,j_{\eta+1,1}^2+\mu)$, hence $\La_\mu$ is differentiable in every $\mu>0$, and we have that $\Lambda'_\mu=0$ if and only if $\pa_\mu G(\La,\mu)=0$ at $(\La_{\mu},\mu)$.\medskip
 
 Assume by contradiction that $\La_\mu$ attains its minimum at $\bar{\mu}>0$ so that $\pa_\mu G(\La,\mu)=0$ at $(\La_{\bar\mu},\bar\mu)$. Fix $\La>\mu$ and define
\[
x(\mu)=\sqrt\mu,
\qquad
y(\mu)=\sqrt{\Lambda-\mu},
\]
so that $x'=(2x)^{-1}$ and $y'=-(2y)^{-1}$. 
Differentiating $G$ we get
\begin{align*}
\partial_\mu G(\Lambda,\mu)
={}&-\frac{1}{2y}I_{\eta}(x)J_{\eta+1}(y)
+\frac{y}{2x}I_{\eta}'(x)J_{\eta+1}(y)
-\frac12 I_{\eta}(x)J_{\eta+1}'(y) \\
&+\frac{1}{2x}J_{\eta}(y)I_{\eta+1}(x)
-\frac{x}{2y}J_{\eta}'(y)I_{\eta+1}(x)
+\frac12 J_{\eta}(y)I_{\eta+1}'(x),
\end{align*}
Substituting the recurrence relations of Bessel functions (see \cite[formulae 9.1.27 and 9.6.26]{ASHandbook})
 \[\begin{split}J'_\nu(y)=\dfrac{\nu}{y}J_\nu(y)-J_{\nu+1}(y),&\qquad J'_{\nu+1}(y)=J_\nu(y)-\dfrac{\nu+1}{y}J_{\nu+1}(y),\\[5pt]
 I'_\nu(y)=\dfrac{\nu}{y}I_\nu(y)+I_{\nu+1}(y),&\qquad I'_{\nu+1}(y)=I_\nu(y)-\dfrac{\nu+1}{y}I_{\nu+1}(y),
 \end{split}\]and simplifying the expression of $\pa_\mu G$, we have
\[
\partial_\mu G(\Lambda,\mu)
=
\frac{x^2+y^2}{2x^2y^2}
\Bigl[
\eta y I_{\eta}(x)J_{\eta+1}(y)
-\eta x J_{\eta}(y)I_{\eta+1}(x)
+xy I_{\eta+1}(x)J_{\eta+1}(y)
\Bigr].
\]
Then, the condition $\partial_\mu G(\Lambda,\mu)=0$ is equivalent to
\begin{equation}
\label{eq:optimality1}
I_{\eta+1}(\sqrt\mu)J_{\eta+1}(\sqrt{\Lambda-\mu})
+\frac{\eta}{\sqrt\mu}I_{\eta}(\sqrt\mu)J_{\eta+1}(\sqrt{\Lambda-\mu})
-\frac{\eta}{\sqrt{\Lambda-\mu}}J_{\eta}(\sqrt{\Lambda-\mu})I_{\eta+1}(\sqrt\mu)=0.
\end{equation}
From $G(\La_\mu,\mu)=0$ we have,
\begin{equation}
\label{eq:optimality2}
\sqrt{\Lambda_\mu-\mu}\,I_{\eta}(\sqrt\mu)J_{\eta+1}\left(\sqrt{\Lambda_\mu-\mu}\right)
=-\sqrt\mu\,J_{\eta}\left(\sqrt{\Lambda_\mu-\mu}\right)I_{\eta+1}(\sqrt\mu),
\end{equation}
so that, substituting in~\eqref{eq:optimality1}, we arrive to
\[
J_{\eta+1}\left(\sqrt{\Lambda_\mu-\mu}\right)
\left[
I_{\eta+1}(\sqrt\mu)+\frac{2\eta}{\sqrt\mu}I_{\eta}(\sqrt\mu)
\right]=0.
\]
Thus, if $\La_\mu$ attains its minimum at $\bar\mu>0$, then 
\[J_{\eta+1}\left(\sqrt{\La_{\bar\mu}-\bar\mu}\right)=0.\] 
Substituting the previous condition in $G(\La,\mu)=0$, however, yields
\[J_\eta\left(\sqrt{\La_{\bar\mu}-\bar\mu}\right)=0,\]
that is, $\sqrt{\La_{\bar\mu}-\bar\mu}$ is a zero of both $J_{\eta+1}$ and $J_\eta$, which contradicts the interlacing property~\eqref{eq:interlacing} of the zeros of Bessel functions. Hence, if $d>1$, $\La_\mu$ attains its minimum at $\bar{\mu}=0$.\medskip

Let 
\[G(\La,\mu)=\sqrt{\La-\mu}\tan(\sqrt{\La-\mu})+\sqrt\mu\tanh(\sqrt\mu),\]
which is strictly increasing in $\La$ for $\sqrt{\La-\mu}\in(\pi/2,\pi)$, so that $\La_\mu'\ge0$ if and only if $\pa_\mu G(\La_\mu,\mu)\le0$. Differentiating with respect to $\mu$ yields
\[\begin{split}\pa_\mu G(\La,\mu)=&-\dfrac{\tan(\sqrt{\La-\mu})+\sqrt{\La-\mu}\left(1+\tan(\sqrt{\La-\mu})^2\right)}{2\sqrt{\La-\mu}}+\dfrac{\tanh(\sqrt{\mu})+\sqrt{\mu}\left(1-\tanh(\sqrt{\mu})^2\right)}{2\sqrt{\mu}}\\[10pt]
=&\dfrac{\sqrt{\mu}\tanh(\sqrt{\mu})}{2\mu}-\dfrac{\sqrt{\La-\mu}\tan(\sqrt{\La-\mu})}{2(\La-\mu)}-\dfrac{1}{2}\left(\tanh(\sqrt\mu)^2+\tan(\sqrt{\La-\mu})^2\right)
\end{split}\]
substituting $G(\La_\mu,\mu)=0$, we have
\[\pa_\mu G(\La_\mu,\mu)= \dfrac{\La_\mu\tanh(\sqrt{\mu})\left(1-\sqrt{\mu}\tanh(\sqrt\mu)\right)}{2\sqrt\mu(\La_\mu-\mu)}.\]
Hence, we have that $\La_\mu'\ge0$ if and only if
\[\sqrt\mu\tanh(\sqrt\mu)\ge1.\]
That is, if $d=1$, by the strict monotonicity of $x\mapsto x\tanh x$, the minimum point $\bar{\mu}$ is the unique solution to
\[\sqrt\mu\tanh(\sqrt\mu)=1.\]

\end{proof}

\begin{proof}[Proof of \autoref{main2}]
    The assertion follows from the previous proposition and the characterization of $\La(B_1\times\R^{n-d})$ given in \autoref{main}.
\end{proof}

\subsubsection*{Acknowledgements} 
We would like to thank Davide Buoso for his thorough read of this work and his fruitful suggestions.

The four authors are members of Gruppo Nazionale per l’Analisi Matematica, la Probabilità e le loro Applicazioni
(GNAMPA) of Istituto Nazionale di Alta Matematica (INdAM). 

The author Paolo Acampora was partially supported by the INdAM - GNAMPA Project, 2026,
”Equazioni alle Derivate Parziali: problemi inversi e stime quantitative”, CUP\_53C25002010001

The author Emanuele Cristoforoni was partially supported by the INdAM - GNAMPA Project, 2025,
”Esistenza, unicità, simmetria e stabilità per problemi ellittici nonlineari e nonlocali”, CUP\_E5324001950001; and by the INdAM - GNAMPA Project, 2026,
” Equazioni nonlocali e nonlineari: alcune questioni di esistenza, rigidita' e regolarita' ”, CUP\_E53C25002010001

\printbibliography[heading=bibintoc]
\Addresses
\end{document}